\documentclass[11pt,fleqn]{article}

\usepackage{amsmath,amssymb,amsthm,esint}
\usepackage{geometry}
\usepackage[pdfpagemode=UseNone,pdfstartview=FitH,hypertex]{hyperref}

\newcommand{\abs}[1]{\left|#1\right|}
\newcommand{\bdry}[1]{\partial #1}
\newcommand{\closure}[1]{\overline{#1}}
\newcommand{\dint}{\ds{\int}}
\newcommand{\dist}[2]{\text{dist}\, (#1,#2)}
\newcommand{\ds}[1]{\displaystyle #1}
\newcommand{\eps}{\varepsilon}
\newcommand{\goodchi}{\protect\raisebox{2pt}{$\chi$}}
\newcommand{\hquad}{\hspace{0.08in}}
\newcommand{\interior}[1]{#1^\circ}
\newcommand{\ip}[2]{\left<#1,#2\right>}
\newcommand{\loc}{\text{loc}}
\newcommand{\norm}[2][]{\left\|#2\right\|_{#1}}
\renewcommand{\O}{\text{O}}
\renewcommand{\o}{\text{o}}
\newcommand{\PS}[1]{$(\text{PS})_{#1}$}
\newcommand{\QED}{\mbox{\qedhere}}
\newcommand{\restr}[2]{\left.#1\right|_{#2}}
\newcommand{\seq}[1]{\left(#1\right)}
\newcommand{\set}[1]{\left\{#1\right\}}

\newcommand{\vol}[1]{\left|#1\right|}
\newcommand{\wstar}{\xrightarrow{w^\ast}}
\newcommand{\wto}{\rightharpoonup}

\newcommand{\B}{{\cal B}}
\newcommand{\D}{{\cal D}}
\newcommand{\M}{{\cal M}}

\newcommand{\R}{\mathbb R}

\DeclareMathOperator{\divg}{div}
\DeclareMathOperator{\supp}{supp}

\newenvironment{enumroman}{\begin{enumerate}

}{\end{enumerate}}

\newtheorem{lemma}{Lemma}[section]
\newtheorem{proposition}[lemma]{Proposition}
\newtheorem{theorem}[lemma]{Theorem}

\theoremstyle{definition}
\newtheorem*{notation}{Notation}

\numberwithin{equation}{section}

\title{\bf Existence and nondegeneracy of ground states in critical free boundary problems}
\author{\bf Yang Yang\thanks{Corresponding author
\newline \indent\; Project supported by NSFC-Tian Yuan Special Foundation (No. 11226116), Natural Science Foundation of Jiangsu Province of China for Young Scholars (No. BK2012109), and the China Scholarship Council (No. 201208320435).}\\
School of Science\\
Jiangnan University\\
Wuxi, 214122, China\\
[\bigskipamount]
\bf Kanishka Perera\\
Department of Mathematical Sciences\\
Florida Institute of Technology\\
Melbourne, FL 32901, USA}
\date{}

\begin{document}

\maketitle

\begin{abstract}
Existence and regularity of minimizers in elliptic free boundary problems have been extensively studied in the literature. The corresponding study of higher critical points was recently initiated in Jerison and Perera \cite{JePe2,JePe}. In particular, the existence and nondegeneracy of a mountain pass point in a superlinear and subcritical free boundary problem related to plasma confinement was proved in \cite{JePe2}. In this paper we study ground states of a critical free boundary problem related to the Br{\'e}zis-Nirenberg problem \cite{MR709644}. We extend the results of \cite{JePe2} to this problem by combining the method introduced there with the concentration compactness principle to overcome the difficulties arising from lack of compactness.
\end{abstract}

\tableofcontents

\section{Introduction and main results}

Existence and regularity of minimizers in elliptic free boundary problems have been extensively studied for over four decades (see, e.g., \cite{MR0440187,MR618549,MR732100,MR861482,MR990856,MR1029856,MR1009785,MR973745,MR1044809,MR1664689,MR1620644,MR1759450,MR1906591,MR2082392,MR2145284,MR2281453,MR2572253} and the references therein). Let $\Omega$ be a bounded domain in $\R^N,\, N \ge 2$ with $\bdry{\Omega} \in C^2$. A typical two-phase free boundary problem seeks a minimizer of the variational integral
\[
\int_\Omega \left[\frac{1}{2}\, |\nabla u|^2 + \goodchi_{\set{u > 0}}(x)\right] dx
\]
among all functions $u \in H^1(\Omega) \cap C(\Omega)$ with prescribed values on some portion of the boundary $\bdry{\Omega}$, where $\goodchi_{\set{u > 0}}$ is the characteristic function of the set $\set{u > 0}$. A local minimizer $u$ satisfies
\[
\Delta u = 0
\]
except on the free boundary $\bdry{\set{u > 0}} \cap \Omega$, and
\[
|\nabla u^+|^2 - |\nabla u^-|^2 = 2
\]
on smooth portions of the free boundary, where $\nabla u^\pm$ are the limits of $\nabla u$ from $\set{u > 0}$ and $\interior{\set{u \le 0}}$, respectively. The existence and regularity of local minimizers for this problem have been studied, for example, in Alt and Caffarelli \cite{MR618549}, Alt, Caffarelli and Friedman \cite{MR732100}, Caffarelli, Jerison and Kenig \cite{MR1906591,MR2082392}, and Weiss \cite{MR1620644,MR1759450}.

The corresponding study of higher critical points was recently initiated in Jerison and Perera \cite{JePe2,JePe}. In particular, the following superlinear and subcritical free boundary problem was studied in \cite{JePe2}:
\[
\left\{\begin{aligned}
- \Delta u & = (u - 1)_+^{p-1} && \text{in } \Omega \setminus \bdry{\set{u > 1}}\\[10pt]
|\nabla u^+|^2 - |\nabla u^-|^2 & = 2 && \text{on } \bdry{\set{u > 1}}\\[10pt]
u & = 0 && \text{on } \bdry{\Omega},
\end{aligned}\right.
\]
where $u_\pm = \max \set{\pm u,0}$ are the positive and negative parts of $u$, respectively, $p > 2$ if $N = 2$ and $2 < p < 2^\ast$ if $N \ge 3$, and $\nabla u^\pm$ are the limits of $\nabla u$ from $\set{u > 1}$ and $\interior{\set{u \le 1}}$, respectively. Here $2^\ast = 2N/(N - 2)$ is the critical Sobolev exponent when $N \ge 3$. The interest in this problem arises from its applications in plasma physics (see, e.g., \cite{MR0412637,MR0602544,MR587175,MR1360544,MR1644436,MR1932180}). The energy functional
\[
J(u) = \int_\Omega \left[\frac{1}{2}\, |\nabla u|^2 + \goodchi_{\set{u > 1}}(x) - \frac{1}{p}\, (u - 1)_+^p\right] dx, \quad u \in H^1_0(\Omega)
\]
associated with this problem is nondifferentiable and therefore standard variational methods cannot be used directly to obtain critical points. A regularization procedure was used in \cite{JePe2} to construct a nontrivial and nondegenerate solution $u$ of mountain pass type that satisfies the equation $- \Delta u = (u - 1)_+^{p-1}$ in the classical sense in $\Omega \setminus \bdry{\set{u > 1}}$, the free boundary condition $|\nabla u^+|^2 - |\nabla u^-|^2 = 2$ in a generalized sense and in the viscosity sense, and vanishes continuously on $\bdry{\Omega}$. Moreover, it was shown in \cite{JePe2} that in a neighborhood of every regular point, the free boundary $\bdry{\set{u > 1}}$ is a $C^{1,\, \alpha}$-surface and hence $u$ satisfies the free boundary condition in the classical sense.

In the present paper we assume that $N \ge 3$ and study the critical free boundary problem
\begin{equation} \label{1.1}
\left\{\begin{aligned}
- \Delta u & = \lambda\, (u - 1)_+ + \kappa\, (u - 1)_+^{2^\ast - 1} && \text{in } \Omega \setminus \bdry{\set{u > 1}}\\[10pt]
|\nabla u^+|^2 - |\nabla u^-|^2 & = 2 && \text{on } \bdry{\set{u > 1}}\\[10pt]
u & = 0 && \text{on } \bdry{\Omega},
\end{aligned}\right.
\end{equation}
where $\lambda, \kappa > 0$ are parameters. Let $\lambda_1 > 0$ be the first eigenvalue of
\[
\left\{\begin{aligned}
- \Delta u & = \lambda u && \text{in } \Omega\\[10pt]
u & = 0 && \text{on } \bdry{\Omega}.
\end{aligned}\right.
\]
Our main result here is that for all $\lambda > \lambda_1$ and sufficiently small $\kappa > 0$, problem \eqref{1.1} has a nontrivial and nondegenerate solution of mountain pass type. This extension of the result in Jerison and Perera \cite{JePe2} to the critical case is nontrivial. Indeed, the noncompactness of the Sobolev imbedding $H^1_0(\Omega) \hookrightarrow L^{2^\ast}(\Omega)$ presents serious new difficulties. We will overcome these difficulties using concentration compactness techniques.

The solution of problem \eqref{1.1} that we construct is a locally Lipschitz continuous function $u$ of class $H^1_0(\Omega) \cap C(\closure{\Omega}) \cap C^1(\closure{\Omega} \setminus \bdry{\set{u > 1}}) \cap C^2(\Omega \setminus \bdry{\set{u > 1}})$ that satisfies the equation $- \Delta u = \lambda\, (u - 1)_+ + \kappa\, (u - 1)_+^{2^\ast - 1}$ in the classical sense in $\Omega \setminus \bdry{\set{u > 1}}$ and vanishes continuously on $\bdry{\Omega}$. It satisfies the free boundary condition in the following generalized sense: for all $\varphi \in C^1_0(\Omega,\R^N)$ such that $u \ne 1$ a.e.\! on the support of $\varphi$,
\begin{equation} \label{1.2}
\lim_{\delta^+ \searrow 0}\, \int_{\set{u = 1 + \delta^+}} \left(|\nabla u|^2 - 2\right) \varphi \cdot n^+\, dS - \lim_{\delta^- \searrow 0}\, \int_{\set{u = 1 - \delta^-}} |\nabla u|^2\, \varphi \cdot n^-\, dS = 0,
\end{equation}
where $n^\pm$ are the outward unit normals to $\bdry{\set{u > 1 \pm \delta^\pm}}$ ($\set{u = 1 \pm \delta^\pm}$ are smooth hypersurfaces for a.a.\! $\delta^\pm > 0$ by Sard's theorem and the above limits are taken through such $\delta^\pm$). In particular, $u$ satisfies the free boundary condition in the classical sense on any smooth portion of the free boundary $\bdry{\set{u > 1}}$. We will refer to such a function $u$ as a generalized solution of problem \eqref{1.1}.

If $u$ is a generalized solution of problem \eqref{1.1}, then by the maximum principle, the set $\set{u < 1}$ is connected and either $u > 0$ everywhere or $u$ vanishes identically. If $u \le 1$ everywhere, then $u$ is harmonic in $\Omega$ and hence vanishes identically again. So if $u$ is a nontrivial solution, then $u > 0$ in $\Omega$ and $u > 1$ in a nonempty open subset of $\Omega$, where it satisfies $- \Delta u = \lambda\, (u - 1) + \kappa\, (u - 1)^{2^\ast - 1}$. Multiplying this equation by $u - 1$ and integrating over the set $\set{u > 1}$ shows that $u$ lies on the Nehari-type manifold
\[
\M = \set{u \in H^1_0(\Omega) : \int_{\set{u > 1}} \Big[|\nabla u|^2 - \lambda\, (u - 1)^2\Big]\, dx = \kappa \int_{\set{u > 1}} (u - 1)^{2^\ast}\, dx > 0}.
\]
The variational functional associated with problem \eqref{1.1} is
\[
J(u) = \int_\Omega \left[\frac{1}{2}\, |\nabla u|^2 + \goodchi_{\set{u > 1}}(x) - \frac{\lambda}{2}\, (u - 1)_+^2 - \frac{\kappa}{2^\ast}\, (u - 1)_+^{2^\ast}\right] dx, \quad u \in H^1_0(\Omega).
\]
For $u \in \M$,
\[
J(u) = \frac{1}{2} \int_{\set{u < 1}} |\nabla u|^2\, dx + \frac{\kappa}{N} \int_{\set{u > 1}} (u - 1)^{2^\ast}\, dx + \vol{\set{u > 1}} > 0,
\]
where $\vol{\cdot}$ denotes the Lebesgue measure in $\R^N$. We will refer to a generalized solution of problem \eqref{1.1} that minimizes $\restr{J}{\M}$ as a ground state.

A generalized solution $u$ of problem \eqref{1.1} is said to be nondegenerate if there exist constants $r_0,\, c > 0$ such that if $x_0 \in \set{u > 1}$ and $r := \dist{x_0}{\set{u \le 1}} \le r_0$, then
\[
u(x_0) \ge 1 + c\, r.
\]
First we prove that ground states are nondegenerate. We have the following theorem.

\begin{theorem} \label{Theorem 1.1}
If $u$ is a locally Lipschitz continuous minimizer of $\restr{J}{\M}$, then $u$ is nondegenerate.
\end{theorem}

We recall that $u \in C(\Omega)$ satisfies the free boundary condition $|\nabla u^+|^2 - |\nabla u^-|^2 = 2$ in the viscosity sense if whenever there exist a point $x_0 \in \bdry{\set{u > 1}}$, a ball $B \subset \set{u > 1}$ (resp. $\interior{\set{u \le 1}}$) with $x_0 \in \bdry{B}$, and $\alpha$ (resp. $\gamma$) $\ge 0$ such that
\[
u(x) \ge 1 + \alpha \ip{x - x_0}{\nu}_+ + \o(|x - x_0|) \hquad (\text{resp. } u(x) \le 1 - \gamma \ip{x - x_0}{\nu}_- + \o(|x - x_0|))
\]
in $B$, where $\nu$ is the interior (resp. exterior) unit normal to $\bdry{B}$ at $x_0$, we have
\[
u(x) < 1 - \gamma \ip{x - x_0}{\nu}_- + \o(|x - x_0|) \hquad (\text{resp. } u(x) > 1 + \alpha \ip{x - x_0}{\nu}_+ + \o(|x - x_0|))
\]
in $B^c$ for any $\gamma$ (resp. $\alpha$) $\ge 0$ such that $\alpha^2 - \gamma^2 >$ (resp. $<$) $2$. Recall also that the point $x_0 \in \bdry{\set{u > 1}}$ is regular if there exists a unit vector $\nu \in \R^N$, called the interior unit normal to the free boundary $\bdry{\set{u > 1}}$ at $x_0$ in the measure theoretic sense, such that
\[
\lim_{r \to 0}\, \frac{1}{r^N} \int_{B_r(x_0)} \abs{\goodchi_{\set{u > 1}}(x) - \goodchi_{\set{\ip{x - x_0}{\nu} > 0}}(x)} dx = 0.
\]
Nondegeneracy will allow us to apply recent regularity results of Lederman and Wolanski \cite{MR2281453} to show that the ground state we construct satisfies the free boundary condition in the viscosity sense, and that near regular points the free boundary is a smooth surface and hence this condition holds in the classical sense.

We turn to constructing a ground state. The functional $J$ has the mountain pass geometry. Let
\[
\Gamma = \set{\gamma \in C([0,1],H^1_0(\Omega)) : \gamma(0) = 0,\, J(\gamma(1)) < 0}
\]
be the class of paths joining $0$ and the set $\set{u \in H^1_0(\Omega) : J(u) < 0}$, and set
\begin{equation} \label{1.3}
c := \inf_{\gamma \in \Gamma}\, \max_{u \in \gamma([0,1])}\, J(u).
\end{equation}
Recall that $u \in H^1_0(\Omega)$ is a mountain pass point of $J$ if the set $\set{v \in U : J(v) < J(u)}$ is neither empty nor path connected for every neighborhood $U$ of $u$ (see Hofer \cite{MR812787}). We will construct a ground state of mountain pass type at the level $c$ when $\lambda > \lambda_1$ and $\kappa > 0$ is sufficiently small. The main existence result of the paper is the following theorem.

\begin{theorem} \label{Theorem 1.2}
Given $\lambda_\ast > \lambda_1$, there exists a constant $\kappa_\ast > 0$, depending only on $\lambda_\ast$ and $\Omega$, such that for all $\lambda \ge \lambda_\ast$ and $0 < \kappa < \kappa_\ast$, there exists a positive ground state solution $u$ of problem \eqref{1.1} satisfying
\begin{enumroman}
\item $u$ is a mountain pass point of $J$ at the level $c > 0$,
\item $u$ is nondegenerate,
\item $u$ satisfies the free boundary condition in the viscosity sense,
\item in a neighborhood of every regular point, the free boundary $\bdry{\set{u > 1}}$ is a $C^{1,\, \alpha}$-surface and $u$ satisfies the free boundary condition in the classical sense.
\end{enumroman}
\end{theorem}

The existence result in the theorem does not follow from a routine application of the mountain pass theorem and the concentration compactness principle due to the lack of smoothness of $J$. Indeed, $J$ is not even continuous, much less of class $C^1$. We will obtain our solution as the limit of mountain pass points of a suitable sequence of $C^1$-functionals approximating $J$ as in Jerison and Perera \cite{JePe2,JePe}. However, we will carry out this regularization procedure for a general nonlinearity that admits critical growth. Consider the free boundary problem
\[
\left\{\begin{aligned}
- \Delta u & = f(u - 1) && \text{in } \Omega \setminus \bdry{\set{u > 1}}\\[10pt]
|\nabla u^+|^2 - |\nabla u^-|^2 & = 2 && \text{on } \bdry{\set{u > 1}}\\[10pt]
u & = 0 && \text{on } \bdry{\Omega},
\end{aligned}\right.
\]
where $f$ is a locally H\"{o}lder continuous function on $\R$ satisfying
\begin{enumerate}
\item[$(f_1)$] $f(t) = 0$ for a.a.\! $x \in \Omega$ and all $t \le 0$,
\item[$(f_2)$] $\exists\, C > 0$ such that $|f(t)| \le C \left(t^{2^\ast - 1} + 1\right)$ for a.a.\! $x \in \Omega$ and all $t > 0$.
\end{enumerate}
The variational functional associated with this problem is
\[
J(u) = \int_\Omega \left[\frac{1}{2}\, |\nabla u|^2 + \goodchi_{\set{u > 1}}(x) - F(u - 1)\right] dx, \quad u \in H^1_0(\Omega),
\]
where $F(t) = \int_0^t f(s)\, ds$. Let $\beta : \R \to [0,2]$ be a smooth function such that $\beta(t) = 0$ for $t \le 0$, $\beta(t) > 0$ for $0 < t < 1$, $\beta(t) = 0$ for $t \ge 1$, and $\int_0^1 \beta(s)\, ds = 1$. Set
\[
\B(t) = \int_0^t \beta(s)\, ds,
\]
and note that $\B : \R \to [0,1]$ is a smooth nondecreasing function such that $\B(t) = 0$ for $t \le 0$, $\B(t) > 0$ for $0 < t < 1$, and $\B(t) = 1$ for $t \ge 1$. For $\eps > 0$, let
\[
J_\eps(u) = \int_\Omega \left[\frac{1}{2}\, |\nabla u|^2 + \B\left(\frac{u - 1}{\eps}\right) - F(u - 1)\right] dx, \quad u \in H^1_0(\Omega)
\]
and note that $J_\eps$ is of class $C^1$. We will prove the following convergence result, which is of independent interest.

\begin{theorem} \label{Theorem 1.3}
Assume $(f_1)$ and $(f_2)$. Let $\eps_j \searrow 0$ and let $u_j$ be a critical point of $J_{\eps_j}$ at the level $c_j$. If $\seq{u_j}$ is bounded in $H^1_0(\Omega) \cap L^\infty(\Omega)$, then there exists a locally Lipschitz continuous function $u \in H^1_0(\Omega) \cap C(\closure{\Omega}) \cap C^1(\closure{\Omega} \setminus \bdry{\set{u > 1}}) \cap C^2(\Omega \setminus \bdry{\set{u > 1}})$ such that, for a renamed subsequence,
\begin{enumroman}
\item \label{Theorem 1.3.i} $u_j \to u$ uniformly on $\closure{\Omega}$,
\item \label{Theorem 1.3.ii} $u_j \to u$ strongly in $H^1_0(\Omega)$,
\item \label{Theorem 1.3.iii} $J(u) \le \liminf c_j \le \limsup c_j \le J(u) + \vol{\set{u = 1}}$, in particular, $u$ is nontrivial if $\limsup c_j > 0$.
\end{enumroman}
Moreover, $u$ satisfies the equation $- \Delta u = f(u - 1)$ in the classical sense in $\Omega \setminus \bdry{\set{u > 1}}$, the free boundary condition $|\nabla u^+|^2 - |\nabla u^-|^2 = 2$ in the generalized sense \eqref{1.2}, and vanishes on $\bdry{\Omega}$. If, in addition, $u$ is nondegenerate, then $u$ satisfies the free boundary condition in the viscosity sense and in a neighborhood of every regular point, the free boundary $\bdry{\set{u > 1}}$ is a $C^{1,\, \alpha}$-surface and $u$ satisfies the free boundary condition in the classical sense.
\end{theorem}

Returning to problem \eqref{1.1}, we have
\[
J_\eps(u) = \int_\Omega \left[\frac{1}{2}\, |\nabla u|^2 + \B\left(\frac{u - 1}{\eps}\right) - \frac{\lambda}{2}\, (u - 1)_+^2 - \frac{\kappa}{2^\ast}\, (u - 1)_+^{2^\ast}\right] dx.
\]
Recall that $J_\eps$ satisfies the Palais-Smale compactness condition at the level $c \in \R$, or \PS{c} condition for short, if every sequence $\seq{u_j} \subset H^1_0(\Omega)$ such that $J_\eps(u_j) \to c$ and $J_\eps'(u_j) \to 0$, called a \PS{c} sequence, has a convergent subsequence. In order to obtain a critical point of $J_\eps$ by applying the mountain pass theorem, we will first show using the concentration compactness principle of Lions \cite{MR834360,MR850686} that given any $M > 0$, there exists $\kappa^\ast > 0$, independent of $\eps$, such that $J_\eps$ satisfies the \PS{c} condition for all $c \le M$ if $0 < \kappa < \kappa^\ast$. Let
\begin{equation} \label{1.4}
S = \inf_{u \in H^1_0(\Omega) \setminus \set{0}}\, \frac{\dint_\Omega |\nabla u|^2\, dx}{\left(\dint_\Omega |u|^{2^\ast}\, dx\right)^{2/2^\ast}}
\end{equation}
be the best Sobolev constant for the imbedding $H^1_0(\Omega) \hookrightarrow L^{2^\ast}(\Omega)$. We have the following proposition.

\begin{proposition} \label{Proposition 1.4}
Given $M > 0$, set
\[
\kappa^\ast = \left[\frac{S^{N/2}}{N(M + \vol{\Omega})}\right]^\frac{2}{N-2}.
\]
If $0 < \kappa < \kappa^\ast$, then $J_\eps$ satisfies the {\em \PS{c}} condition for all $c \le M$.
\end{proposition}

In order to apply Theorem \ref{Theorem 1.3}, we also need uniform $H^1$ and $L^\infty$ estimates for critical points of $J_\eps$ below a given level. If $u$ is a critical point of $J_\eps$ with $J_\eps(u) \le M$, then
\begin{equation} \label{1.5}
\int_\Omega (u^+)^{2^\ast} dx \le \frac{N(M + \vol{\Omega})}{\kappa}
\end{equation}
by Lemma \ref{Lemma 4.1} in Section \ref{Section 4}, which together with the H\"{o}lder inequality and $J_\eps(u) \le M$ gives a bound for $\norm{u}$ that depends only on $N$, $M$, $\vol{\Omega}$, $\kappa$ and $\lambda$. We will prove that there exists $\kappa_\ast > 0$, independent of $\eps$, such that critical points of $J_\eps$ in $\set{J_\eps(u) \le M}$ are also bounded in $L^\infty(\Omega)$ if $0 < \kappa < \kappa_\ast$. We have the following proposition.

\begin{proposition} \label{Proposition 1.5}
Given $M > 0$, set
\[
\kappa_\ast = \left(1 - \frac{4}{N^2}\right)^\frac{N}{N-2} \kappa^\ast,
\]
where $\kappa^\ast$ is as in Proposition \ref{Proposition 1.4}. If $0 < \kappa < \kappa_\ast$, then there exists a constant $C > 0$, depending on $M$, $\Omega$, $\kappa$ and $\lambda$, but not on $\eps > 0$, such that
\[
\norm[\infty]{u} \le C
\]
whenever $u$ is a critical point of $J_\eps$ with $J_\eps(u) \le M$.
\end{proposition}

We would like to emphasize that the novelty here is that we extend the results of \cite{JePe2} to the critical exponent case, where the noncompactness of the Sobolev imbedding presents serious new difficulties, by combining the method introduced there with concentration compactness techniques. Many of the arguments in this paper are adapted from \cite{JePe2}. They are repeated here in detail for the sake of completeness and for the convenience of the reader.

In closing the introduction we note that problem \eqref{1.1} is also related to the well-studied Br{\'e}zis-Nirenberg problem
\[
\left\{\begin{aligned}
- \Delta u & = \lambda u + u^{2^\ast - 1} && \text{in } \Omega\\[10pt]
u & > 0 && \text{in } \Omega\\[10pt]
u & = 0 && \text{on } \bdry{\Omega}.
\end{aligned}\right.
\]
In a celebrated paper \cite{MR709644}, Br{\'e}zis and Nirenberg proved, among other things, that this problem has a solution if $0 < \lambda < \lambda_1$ and $N \ge 4$, and no solution if $\lambda \ge \lambda_1$. This pioneering work has stimulated a vast literature (see, e.g., \cite{MR779872,MR831041,MR829403,MR1009077,MR1083144,MR1154480,MR1306583,MR1473856,MR1441856,MR1491613,MR1695021,MR1784441,MR1961520} and the references therein).

\begin{notation}
Throughout the paper we write
\[
u = (u - 1)_+ + [1 - (u - 1)_-] =: u^+ + u^-.
\]
\end{notation}

\section{Nondegeneracy of ground states} \label{Section 2}

In this section we prove that ground states are nondegenerate and those that are at the level $c$ are mountain pass points of $J$, where $c$ is as in \eqref{1.3}. The manifold $\M$ is contained in the set
\[
U = \set{u \in H^1_0(\Omega) : u^\pm \ne 0}.
\]
For $u \in U$, consider the curve
\[
\zeta_u(s) = \begin{cases}
(1 + s)\, u^-, & s \in [-1,0]\\[5pt]
u^- + s\, u^+, & s \in (0,\infty),
\end{cases}
\]
which passes through $u$ at $s = 1$. For $s \in [-1,0]$,
\[
J(\zeta_u(s)) = \frac{(1 + s)^2}{2} \int_{\set{u < 1}} |\nabla u|^2\, dx
\]
is increasing in $s$. For $s \in (0,\infty)$,
\begin{multline} \label{2.1}
J(\zeta_u(s)) = \frac{1}{2} \int_{\set{u < 1}} |\nabla u|^2\, dx + \frac{s^2}{2} \int_{\set{u > 1}} \Big[|\nabla u|^2 - \lambda\, (u - 1)^2\Big]\, dx\\[7.5pt]
- \frac{\kappa s^{2^\ast}}{2^\ast} \int_{\set{u > 1}} (u - 1)^{2^\ast}\, dx + \vol{\set{u > 1}}.
\end{multline}
Hence
\[
\lim_{s \searrow 0}\, J(\zeta_u(s)) = J(\zeta_u(0)) + \vol{\set{u > 1}} > J(\zeta_u(0))
\]
and
\[
\frac{d}{ds}\, J(\zeta_u(s)) = s \left[\int_{\set{u > 1}} \Big[|\nabla u|^2 - \lambda\, (u - 1)^2\Big]\, dx - \kappa s^\frac{4}{N-2} \int_{\set{u > 1}} (u - 1)^{2^\ast}\, dx\right].
\]
Set
\[
s_u = \left[\frac{\dint_{\set{u > 1}} \Big[|\nabla u|^2 - \lambda\, (u - 1)^2\Big]\, dx}{\kappa \dint_{\set{u > 1}} (u - 1)^{2^\ast}\, dx}\right]^\frac{N-2}{4}.
\]
Then $J(\zeta_u(s))$ increases for $s \in (0,s_u)$, attains its maximum at $s = s_u$, decreases for $s \in (s_u,\infty)$, and
\begin{equation} \label{2.2}
\lim_{s \to \infty}\, J(\zeta_u(s)) = - \infty.
\end{equation}

\begin{proposition} \label{Proposition 2.1}
Let $c$ be as in \eqref{1.3}. Then
\begin{equation} \label{2.3}
c \le \inf_{u \in \M}\, J(u).
\end{equation}
If $u \in \M$ and $J(u) = c$, then $u$ is a mountain pass point of $J$. In particular, ground states at the level $c$ are mountain pass points of $J$.
\end{proposition}

\begin{proof}
For each $u \in \M$, there exists $s_0 > 1$ such that $J(\zeta_u(s_0)) < 0$ by \eqref{2.2}. Then
\[
\gamma_u(t) = \zeta_u((s_0 + 1)\, t - 1), \quad t \in [0,1]
\]
defines a path $\gamma_u \in \Gamma$ such that
\[
\max_{v \in \gamma_u([0,1])}\, J(v) = J(u),
\]
so $c \le J(u)$. \eqref{2.3} follows.

Now suppose $J(u) = c$ and let $U$ be a neighborhood of $u$. The path $\gamma_u$ passes through $u$ at $t = 2/(s_0 + 1) =: t_0$ and $J(\gamma_u(t)) < c$ for $t \ne t_0$. By the continuity of $\gamma_u$, there exist $0 < t^- < t_0 < t^+ < 1$ such that $\gamma_u(t^\pm) \in U$, in particular, the set $\set{v \in U : J(v) < c}$ is nonempty. If it is path connected, then it contains a path $\eta$ joining $\gamma_u(t^\pm)$, and reparametrizing $\restr{\gamma_u}{[0,t^-]} \cup \eta \cup \restr{\gamma_u}{[t^+,1]}$ gives a path in $\Gamma$ on which $J < c$, contradicting the definition of $c$. So it is not path connected, and $u$ is a mountain pass point of $J$.
\end{proof}

For $u \in U$, $\zeta_u$ intersects $\M$ exactly at one point, namely, where $s = s_u$, and $s_u = 1$ if $u \in \M$. So we can define a (nonradial) continuous projection $\pi : U \to \M$ by
\[
\pi(u) = \zeta_u(s_u) = u^- + s_u\, u^+.
\]
Since
\[
s_u^2 \int_{\set{u > 1}} \Big[|\nabla u|^2 - \lambda\, (u - 1)^2\Big]\, dx = \kappa s_u^{2^\ast} \int_{\set{u > 1}} (u - 1)^{2^\ast}\, dx,
\]
we have
\begin{equation} \label{2.4}
J(\pi(u)) = \frac{1}{2} \int_{\set{u < 1}} |\nabla u|^2\, dx + \frac{s_u^2}{N} \int_{\set{u > 1}} \Big[|\nabla u|^2 - \lambda\, (u - 1)^2\Big]\, dx + \vol{\set{u > 1}}
\end{equation}
by \eqref{2.1}, in particular,
\begin{equation} \label{2.5}
J(u) = \frac{1}{2} \int_{\set{u < 1}} |\nabla u|^2\, dx + \frac{1}{N} \int_{\set{u > 1}} \Big[|\nabla u|^2 - \lambda\, (u - 1)^2\Big]\, dx + \vol{\set{u > 1}}
\end{equation}
for $u \in \M$. We are now ready to prove Theorem \ref{Theorem 1.1}.

\begin{proof}[Proof of Theorem \ref{Theorem 1.1}]
Let $x_0 \in \set{u > 1}$ and let $r := \dist{x_0}{\set{u \le 1}} \le 1$. Then let $y = (x - x_0)/r$, let $V = \set{y : x \in \Omega}$, and set
\[
v(y) = \frac{u(x) - 1}{r}, \quad y \in V.
\]
Since $u(x) > 1$ for all $x \in B_r(x_0)$ and there exists a point $x_1 \in \bdry{B_r(x_0)}$ such that $u(x_1) = 1$,
\begin{equation} \label{2.6}
0 < v(y) = \frac{u(x) - u(x_1)}{r} \le \frac{L\, |x - x_1|}{r} < 2L \quad \forall y \in B_1(0),
\end{equation}
where $L > 0$ is the Lipschitz coefficient of $u$ in $\set{u \ge 1}$. Then
\[
- \Delta v = \lambda r^2\, v + \kappa r^{2^\ast} v^{2^\ast - 1} \quad \text{in } B_1(0),
\]
so by the Harnack inequality there exists a constant $C > 0$ such that
\[
v(y) \le C\, (\alpha + r^2) \quad \forall y \in B_{2/3}(0),
\]
where
\[
\alpha = v(0) = \frac{u(x_0) - 1}{r}.
\]

Take a smooth cutoff function $\psi : B_1(0) \to [0,1]$ such that $\psi = 0$ in $\closure{B_{1/3}(0)}$, $0 < \psi < 1$ in $B_{2/3}(0) \setminus \closure{B_{1/3}(0)}$ and $\psi = 1$ in $B_1(0) \setminus B_{2/3}(0)$, let
\[
w(y) = \begin{cases}
\min \set{v(y),C\, (\alpha + r^2)\, \psi(y)}, & y \in B_{2/3}(0)\\[7.5pt]
v(y), & y \in V \setminus B_{2/3}(0),
\end{cases}
\]
and set $z(x) = r w(y) + 1$. Since $\set{z > 1} = \set{u > 1} \setminus \closure{B_{r/3}(x_0)}$, $z = 1$ in $\closure{B_{r/3}(x_0)}$, and $z = u$ outside $\D := \set{x \in B_{2r/3}(x_0) : v(y) \ge C\, (\alpha + r^2)\, \psi(y)}$,
\begin{eqnarray*}
s_z^\frac{4}{N-2} & = & \frac{\dint_{\set{z > 1}} \Big[|\nabla z|^2 - \lambda\, (z - 1)^2\Big]\, dx}{\kappa \dint_{\set{z > 1}} (z - 1)^{2^\ast}\, dx}\\[10pt]
& = & \frac{\dint_{\set{u > 1}} \Big[|\nabla z|^2 - \lambda\, (z - 1)^2\Big]\, dx}{\kappa \dint_{\set{u > 1}} (z - 1)^{2^\ast}\, dx}\\[10pt]
& = & \frac{\dint_{\set{u > 1} \setminus \D} \Big[|\nabla u|^2 - \lambda\, (u - 1)^2\Big]\, dx + \dint_\D \Big[|\nabla z|^2 - \lambda\, (z - 1)^2\Big]\, dx}{\kappa \left(\dint_{\set{u > 1} \setminus \D} (u - 1)^{2^\ast}\, dx + \dint_\D (z - 1)^{2^\ast}\, dx\right)}\\[10pt]
\end{eqnarray*}
\begin{eqnarray*}
\phantom{s_z^\frac{4}{N-2}} & \le & \frac{\dint_{\set{u > 1}} \Big[|\nabla u|^2 - \lambda\, (u - 1)^2\Big]\, dx + \dint_\D \Big[|\nabla z|^2 + \lambda\, (u - 1)^2\Big]\, dx}{\kappa \left(\dint_{\set{u > 1}} (u - 1)^{2^\ast}\, dx - \dint_\D (u - 1)^{2^\ast}\, dx\right)}\\[10pt]
& = & \frac{1 + \dfrac{\dint_\D \Big[|\nabla z|^2 + \lambda\, (u - 1)^2\Big]\, dx}{\kappa \dint_{\set{u > 1}} (u - 1)^{2^\ast}\, dx}}{1 - \dfrac{\dint_\D (u - 1)^{2^\ast}\, dx}{\dint_{\set{u > 1}} (u - 1)^{2^\ast}\, dx}}.
\end{eqnarray*}
We have
\begin{equation} \label{2.7}
\int_\D |\nabla z|^2\, dx = C^2\, (\alpha + r^2)^2\, r^N \int_{\set{y : x \in \D}} |\nabla \psi|^2\, dy = \O(r^N) \quad \text{as } r \to 0
\end{equation}
since $0 < \alpha < 2L$ by \eqref{2.6}. As in \eqref{2.6}, $0 < u - 1 < 2L\, r$ in $\D$, and $\vol{\D} = \O(r^N)$, so
\begin{equation} \label{2.8}
\int_\D (u - 1)^2\, dx = \O(r^{2+N}), \qquad \int_\D (u - 1)^{2^\ast}\, dx = \O(r^{2^\ast + N}).
\end{equation}
It follows that
\begin{equation} \label{2.9}
s_z^2 \le 1 + \frac{N - 2}{2 \kappa}\, \frac{\dint_\D |\nabla z|^2\, dx}{\dint_{\set{u > 1}} (u - 1)^{2^\ast}\, dx} + \O(r^{2+N}).
\end{equation}

Since $u$ is a minimizer of $\restr{J}{\M}$,
\[
J(u) \le J(\pi(z)).
\]
Since $z^- = u^-$, $\set{z > 1} = \set{u > 1} \setminus \closure{B_{r/3}(x_0)}$, and $z = 1$ in $\closure{B_{r/3}(x_0)}$, this inequality reduces to
\[
\frac{1}{N} \int_{\set{u > 1}} \Big[|\nabla u|^2 - \lambda\, (u - 1)^2\Big]\, dx + \vol{B_{r/3}(x_0)} \le \frac{s_z^2}{N} \int_{\set{u > 1}} \Big[|\nabla z|^2 - \lambda\, (z - 1)^2\Big]\, dx
\]
by \eqref{2.4} and \eqref{2.5}. Since $z = u$ in $\set{u > 1} \setminus \D$, this implies
\begin{eqnarray*}
N \vol{B_{1/3}(0)} r^N & \le & s_z^2 \left(\int_\D \Big[|\nabla z|^2 - \lambda\, (z - 1)^2\Big]\, dx + \int_{\set{u > 1} \setminus \D} \Big[|\nabla u|^2 - \lambda\, (u - 1)^2\Big]\, dx\right)\\[7.5pt]
& & - \int_{\set{u > 1}} \Big[|\nabla u|^2 - \lambda\, (u - 1)^2\Big]\, dx\\[10pt]
& = & s_z^2 \int_\D \Big[|\nabla z|^2 - \lambda\, (z - 1)^2 - |\nabla u|^2 + \lambda\, (u - 1)^2\Big]\, dx\\[7.5pt]
& & + \left(s_z^2 - 1\right) \int_{\set{u > 1}} \Big[|\nabla u|^2 - \lambda\, (u - 1)^2\Big]\, dx\\[10pt]
& \le & s_z^2 \int_\D \Big[|\nabla z|^2 + \lambda\, (u - 1)^2\Big]\, dx + \kappa \left(s_z^2 - 1\right) \int_{\set{u > 1}} (u - 1)^{2^\ast}\, dx\\[10pt]
& \le & \frac{N}{2} \int_\D |\nabla z|^2\, dx + \O(r^{2+N})
\end{eqnarray*}
by \eqref{2.8} and \eqref{2.9}. In view of the first equality in \eqref{2.7}, this gives $r_0,\, c > 0$ such that $r \le r_0$ implies $\alpha \ge c$, which is the desired conclusion.
\end{proof}

\section{A general convergence result}

In this section we prove Theorem \ref{Theorem 1.3}. If $u$ is a critical point of $J_\eps$, then $u$ is a weak solution of
\[
\left\{\begin{aligned}
- \Delta u & = - \frac{1}{\eps}\, \beta\left(\frac{u - 1}{\eps}\right) + f(u - 1) && \text{in } \Omega\\[10pt]
u & = 0 && \text{on } \bdry{\Omega},
\end{aligned}\right.
\]
and hence also a classical solution by elliptic regularity theory. By the maximum principle, $u \ge 0$ in $\Omega$. The crucial ingredient in the passage to the limit is the following uniform Lipschitz continuity result, proved in Caffarelli, Jerison, and Kenig \cite{MR1906591}.

\begin{proposition}[{\cite[Theorem 5.1]{MR1906591}}] \label{Proposition 3.1}
Suppose that $u$ is a Lipschitz continuous function on $B_1(0) \subset \R^N$ satisfying the distributional inequalities
\[
\pm \Delta u \le A \left(\frac{1}{\eps}\, \goodchi_{\set{|u - 1| < \eps}}(x) + 1\right)
\]
for some constants $A > 0$ and $0 < \eps \le 1$. Then there exists a constant $C > 0$, depending on $N$, $A$ and $\int_{B_1(0)} u^2\, dx$, but not on $\eps$, such that
\[
\max_{x \in B_{1/2}(0)}\, |\nabla u(x)| \le C.
\]
\end{proposition}

Let $\eps_j$ and $u_j$ be as in Theorem \ref{Theorem 1.3}. We may assume that $0 < \eps_j \le 1$. Since $\seq{u_j}$ is bounded in $L^\infty(\Omega)$,
\begin{equation} \label{3.1}
|f(u_j(x) - 1)| \le A_0 \quad \text{for a.a.\! } x \in \Omega
\end{equation}
for some constant $A_0 > 0$ by $(f_1)$ and $(f_2)$. Then
\[
\pm \Delta u_j = \pm \frac{1}{\eps_j}\, \beta\left(\frac{u_j - 1}{\eps_j}\right) \mp f(u_j - 1) \le \frac{2}{\eps_j}\, \goodchi_{\set{|u_j - 1| < \eps_j}}(x) + A_0
\]
since $0 \le \beta \le 2\, \goodchi_{(-1,1)}$. Since $\seq{u_j}$ is also bounded in $L^2(\Omega)$, it follows from Proposition \ref{Proposition 3.1} that for each $r > 0$, there exists a constant $C(r) > 0$ such that
\begin{equation} \label{3.2}
\max_{x \in B_{r/2}(x_0)}\, |\nabla u_j(x)| \le C(r)
\end{equation}
whenever $B_r(x_0) \subset \Omega$. We are now ready to prove Theorem \ref{Theorem 1.3}.

\begin{proof}[Proof of Theorem \ref{Theorem 1.3}]
Since $\seq{u_j}$ is bounded in $H^1_0(\Omega)$, a renamed subsequence converges weakly in $H^1_0(\Omega)$ to some $u$. Since $\seq{u_j}$ is bounded in $L^\infty(\Omega)$ and uniformly locally Lipschitz continuous by \eqref{3.2}, a standard diagonalization argument shows that $u_j \to u$ uniformly on compact subsets of $\Omega$ for a further subsequence, and $u$ is locally Lipschitz continuous. Let $\varphi_0 > 0$ be the solution of
\[
\left\{\begin{aligned}
- \Delta \varphi_0 & = A_0 && \text{in } \Omega\\[10pt]
\varphi_0 & = 0 && \text{on } \bdry{\Omega},
\end{aligned}\right.
\]
where $A_0$ is as in \eqref{3.1}. Since
\begin{equation} \label{3.3}
- \Delta u_j \le A_0 \quad \text{in } \Omega,
\end{equation}
we have
\begin{equation} \label{3.4}
0 \le u_j(x) \le \varphi_0(x) \quad \forall x \in \Omega
\end{equation}
by the maximum principle, so $0 \le u \le \varphi_0$. Hence $|u_j - u| \le \varphi_0$. It follows that $u$ continuously extends to $\closure{\Omega}$ with zero boundary values and $u_j \to u$ uniformly on $\closure{\Omega}$.

Next we show that $u$ satisfies the equation $- \Delta u = f(u - 1)$ in $\set{u \ne 1}$. Let $\varphi \in C^\infty_0(\set{u > 1})$. Then $u \ge 1 + 2\, \eps$ on the support of $\varphi$ for some $\eps > 0$. For all sufficiently large $j$, $\eps_j < \eps$ and $|u_j - u| < \eps$ in $\Omega$. Then $u_j \ge 1 + \eps_j$ on the support of $\varphi$, so testing
\begin{equation} \label{3.5}
- \Delta u_j = - \frac{1}{\eps_j}\, \beta\left(\frac{u_j - 1}{\eps_j}\right) + f(u_j - 1)
\end{equation}
with $\varphi$ gives
\[
\int_\Omega \nabla u_j \cdot \nabla \varphi\, dx = \int_\Omega f(u_j - 1)\, \varphi\, dx.
\]
Since $u_j \to u$ weakly in $H^1_0(\Omega)$ and uniformly on $\Omega$, passing to the limit gives
\begin{equation} \label{3.6}
\int_\Omega \nabla u \cdot \nabla \varphi\, dx = \int_\Omega f(u - 1)\, \varphi\, dx.
\end{equation}
This then holds for all $\varphi \in H^1_0(\set{u > 1})$ by density, and hence $u$ is a classical solution of $- \Delta u = f(u - 1)$ in $\set{u > 1}$. A similar argument shows that $u$ is harmonic in $\set{u < 1}$.

Now we show that $u$ is harmonic in $\interior{\set{u \le 1}}$. Testing \eqref{3.3} with any nonnegative $\varphi \in C^\infty_0(\Omega)$ gives
\[
\int_\Omega \nabla u_j \cdot \nabla \varphi\, dx \le A_0 \int_\Omega \varphi\, dx,
\]
and passing to the limit gives
\[
\int_\Omega \nabla u \cdot \nabla \varphi\, dx \le A_0 \int_\Omega \varphi\, dx,
\]
so
\begin{equation} \label{3.7}
- \Delta u \le A_0 \quad \text{in } \Omega
\end{equation}
in the weak sense. On the other hand, since $u$ is harmonic in $\set{u < 1}$, $\kappa := \Delta (u - 1)_-$ is a nonnegative Radon measure supported on $\Omega \cap \bdry{\set{u < 1}}$ by Alt and Caffarelli \cite[Remark 4.2]{MR618549}, so
\begin{equation} \label{3.8}
- \Delta u = \kappa \ge 0 \quad \text{in } \set{u \le 1}.
\end{equation}
It follows from \eqref{3.7} and \eqref{3.8} that $u \in W^{2,\, q}_\loc(\interior{\set{u \le 1}}),\, 1 < q < \infty$ and hence $\kappa$ is actually supported on $\Omega \cap \bdry{\set{u < 1}} \cap \bdry{\set{u > 1}}$, so $u$ is harmonic in $\interior{\set{u \le 1}}$.

We are now ready to prove \ref{Theorem 1.3.ii}. By standard regularity arguments, $u_j \in C^1(\closure{\Omega})$, $u \in C^1(\closure{\Omega} \setminus \bdry{\set{u > 1}})$, and $\partial u_j/\partial n \to \partial u/\partial n$ on $\bdry{\Omega}$, where $n$ is the outward unit normal. Multiplying \eqref{3.5} by $u_j - 1$, integrating by parts, and noting that $\beta((t - 1)/\eps_j)\, (t - 1) \ge 0$ for all $t$ gives
\begin{multline} \label{3.9}
\int_\Omega |\nabla u_j|^2\, dx \le \int_\Omega f(u_j - 1)\, (u_j - 1)\, dx - \int_{\bdry{\Omega}} \frac{\partial u_j}{\partial n}\, dS\\[7.5pt]
\to \int_\Omega f(u - 1)\, (u - 1)\, dx - \int_{\bdry{\Omega}} \frac{\partial u}{\partial n}\, dS
\end{multline}
since $\partial u_j/\partial n \to \partial u/\partial n$ pointwise and $\partial \varphi_0/\partial n \le \partial u_j/\partial n \le 0$ by \eqref{3.4}. Fix $0 < \eps < 1$. Taking $\varphi = (u - 1 - \eps)_+$ in \eqref{3.6} gives
\begin{equation} \label{3.10}
\int_{\set{u > 1 + \eps}} |\nabla u|^2\, dx = \int_\Omega f(u - 1)\, (u - 1 - \eps)_+\, dx,
\end{equation}
and integrating $(u - 1 + \eps)_-\, \Delta u = 0$ over $\Omega$ gives
\begin{equation} \label{3.11}
\int_{\set{u < 1 - \eps}} |\nabla u|^2\, dx = - (1 - \eps) \int_{\bdry{\Omega}} \frac{\partial u}{\partial n}\, dS.
\end{equation}
Adding \eqref{3.10} and \eqref{3.11}, and letting $\eps \searrow 0$ gives
\[
\int_\Omega |\nabla u|^2\, dx = \int_\Omega f(u - 1)\, (u - 1)\, dx - \int_{\bdry{\Omega}} \frac{\partial u}{\partial n}\, dS,
\]
which together with \eqref{3.9} gives
\[
\limsup_{j \to \infty}\, \int_\Omega |\nabla u_j|^2\, dx \le \int_\Omega |\nabla u|^2\, dx.
\]
Since $u_j \rightharpoonup u$ in $H^1_0(\Omega)$ and hence $\norm{u} \le \liminf \norm{u_j}$, it follows that $\norm{u_j} \to \norm{u}$, so $u_j \to u$ strongly in $H^1_0(\Omega)$.

To prove \ref{Theorem 1.3.iii}, write
\begin{multline*}
J_{\eps_j}(u_j) = \int_\Omega \left[\frac{1}{2}\, |\nabla u_j|^2 + \B\left(\frac{u_j - 1}{\eps_j}\right) \goodchi_{\set{u \ne 1}}(x) - F(u_j - 1)\right] dx\\[7.5pt]
+ \int_{\set{u = 1}} \B\left(\frac{u_j - 1}{\eps_j}\right) dx.
\end{multline*}
Since $\B((u_j - 1)/\eps_j)\, \goodchi_{\set{u \ne 1}}$ converges pointwise to $\goodchi_{\set{u > 1}}$ and is bounded by $1$, the first integral converges to $J(u)$ by \ref{Theorem 1.3.i} and \ref{Theorem 1.3.ii}, and
\[
0 \le \int_{\set{u = 1}} \B\left(\frac{u_j - 1}{\eps_j}\right) dx \le \vol{\set{u = 1}}
\]
since $0 \le \B \le 1$, so \ref{Theorem 1.3.iii} follows.

\newpage

Now we show that $u$ satisfies the generalized free boundary condition, i.e., for all $\varphi \in C^1_0(\Omega,\R^N)$ such that $u \ne 1$ a.e.\! on the support of $\varphi$,
\begin{equation} \label{3.12}
\lim_{\delta^+ \searrow 0}\, \int_{\set{u = 1 + \delta^+}} \left(2 - |\nabla u|^2\right) \varphi \cdot n\, dS - \lim_{\delta^- \searrow 0}\, \int_{\set{u = 1 - \delta^-}} |\nabla u|^2\, \varphi \cdot n\, dS = 0,
\end{equation}
where $n$ is the outward unit normal to $\set{1 - \delta^- < u < 1 + \delta^+}$. Multiplying \eqref{3.5} by $\nabla u_j \cdot \varphi$ and integrating over $\set{1 - \delta^- < u < 1 + \delta^+}$ gives
\begin{eqnarray*}
0 & = & \int_{\set{1 - \delta^- < u < 1 + \delta^+}} \left[- \Delta u_j + \frac{1}{\eps_j}\, \beta\left(\frac{u_j - 1}{\eps_j}\right) - f(u_j - 1)\right] \nabla u_j \cdot \varphi\, dx\\[10pt]
& = & \int_{\set{1 - \delta^- < u < 1 + \delta^+}} \bigg[\divg \left(\frac{1}{2}\, |\nabla u_j|^2\, \varphi - (\nabla u_j \cdot \varphi)\, \nabla u_j\right) + \nabla u_j\, D\varphi \cdot \nabla u_j\\[10pt]
&& - \frac{1}{2}\, |\nabla u_j|^2\, \divg \varphi + \nabla \B\left(\frac{u_j - 1}{\eps_j}\right) \cdot \varphi - \nabla F(u_j - 1) \cdot \varphi\bigg]\, dx\\[10pt]
& = & \frac{1}{2} \int_{\set{u = 1 + \delta^+} \bigcup \set{u = 1 - \delta^-}} \left[|\nabla u_j|^2\, \varphi - 2\, (\nabla u_j \cdot \varphi)\, \nabla u_j + 2\, \B\left(\frac{u_j - 1}{\eps_j}\right) \varphi\right] \cdot n\, dS\\[10pt]
&& - \int_{\set{u = 1 + \delta^+} \bigcup \set{u = 1 - \delta^-}} F(u_j - 1)\, \varphi \cdot n\, dS\\[10pt]
&& - \int_{\set{1 - \delta^- < u < 1 + \delta^+}} \left[\B\left(\frac{u_j - 1}{\eps_j}\right) - F(u_j - 1)\right] \divg \varphi\, dx\\[10pt]
&& + \int_{\set{1 - \delta^- < u < 1 + \delta^+}} \left(\nabla u_j\, D\varphi \cdot \nabla u_j - \frac{1}{2}\, |\nabla u_j|^2\, \divg \varphi\right) dx\\[10pt]
& =: & \frac{I_1}{2} - I_2 - I_3 + I_4.
\end{eqnarray*}
Since $u_j \to u$ uniformly on $\closure{\Omega}$, strongly in $H^1_0(\Omega)$, and locally in $C^1(\closure{\Omega} \setminus \bdry{\set{u > 1}})$,
\begin{eqnarray*}
I_1 & \to & \int_{\set{u = 1 + \delta^+} \bigcup \set{u = 1 - \delta^-}} \left(|\nabla u|^2\, \varphi - 2\, (\nabla u \cdot \varphi)\, \nabla u\right) \cdot n\, dS + \int_{\set{u = 1 + \delta^+}} 2\, \varphi \cdot n\, dS\\[10pt]
& = & \int_{\set{u = 1 + \delta^+}} \left(2 - |\nabla u|^2\right) \varphi \cdot n\, dS - \int_{\set{u = 1 - \delta^-}} |\nabla u|^2\, \varphi \cdot n\, dS
\end{eqnarray*}
since $n = \pm \nabla u/|\nabla u|$ on $\set{u = 1 \pm \delta^\pm}$, and
\[
I_2 \to \int_{\set{u = 1 + \delta^+}} F(u - 1)\, \varphi \cdot n\, dS = F(\delta^+) \int_{\set{u = 1 + \delta^+}} \varphi \cdot n\, dS = F(\delta^+) \int_{\set{u < 1 + \delta^+}} \divg \varphi\, dx,
\]
which goes to zero as $\delta^+ \searrow 0$ by $(f_2)$. Since $0 \le \B((u_j - 1)/\eps_j) \le 1$,
\begin{eqnarray*}
|I_3| & \le & \int_{\set{1 - \delta^- < u < 1 + \delta^+}} \big(1 + |F(u_j - 1)|\big)\, |\divg \varphi|\, dx\\[10pt]
& \to & \int_{\set{1 - \delta^- < u < 1 + \delta^+}} \big(1 + |F(u - 1)|\big)\, |\divg \varphi|\, dx,
\end{eqnarray*}
and
\[
I_4 \to \int_{\set{1 - \delta^- < u < 1 + \delta^+}} \left(\nabla u\, D\varphi \cdot \nabla u - \frac{1}{2}\, |\nabla u|^2\, \divg \varphi\right) dx.
\]
The last two integrals go to zero as $\delta^\pm \searrow 0$ since $\vol{\set{u = 1} \cap \supp \varphi} = 0$, so first letting $j \to \infty$ and then letting $\delta^\pm \searrow 0$ gives \eqref{3.12}.

Since $f(u_j - 1) \wstar f(u - 1)$ in $L^\infty(\Omega)$, the rest follows from Corollaries 7.1, 7.2 and Theorem 9.2 of Lederman and Wolanski \cite{MR2281453}.
\end{proof}

\section{Compactness and $L^\infty$ estimates} \label{Section 4}

In this section we prove Propositions \ref{Proposition 1.4} and \ref{Proposition 1.5}. We have
\[
J_\eps(u) = \int_\Omega \left[\frac{1}{2} \left(|\nabla u^+|^2 + |\nabla u^-|^2\right) + \B\left(\frac{u^+}{\eps}\right) - \frac{\lambda}{2}\, (u^+)^2 - \frac{\kappa}{2^\ast}\, (u^+)^{2^\ast}\right] dx
\]
and
\[
J_\eps'(u)\, v = \int_\Omega \left[\left(\nabla u^+ + \nabla u^-\right) \cdot \nabla v + \beta\left(\frac{u^+}{\eps}\right) \frac{v}{\eps} - \lambda\, u^+ v - \kappa\, (u^+)^{2^\ast - 1} v\right] dx, \quad v \in H^1_0(\Omega).
\]
We will make repeated use of the following simple lemma.

\begin{lemma} \label{Lemma 4.1}
For all $u \in H^1_0(\Omega)$,
\[
\frac{\kappa}{N} \int_\Omega (u^+)^{2^\ast} dx \le J_\eps(u) - \frac{1}{2}\, J_\eps'(u)\, u^+ + \vol{\Omega}.
\]
\end{lemma}

\begin{proof}
Since $\B(t) \ge 0$ for all $t$,
\[
J_\eps(u) \ge \int_\Omega \left[\frac{1}{2}\, |\nabla u^+|^2 - \frac{\lambda}{2}\, (u^+)^2 - \frac{\kappa}{2^\ast}\, (u^+)^{2^\ast}\right] dx,
\]
and since $\beta(t)\, t \le 2$ for all $t$,
\[
J_\eps'(u)\, u^+ \le \int_\Omega \left[|\nabla u^+|^2 - \lambda\, (u^+)^2 - \kappa\, (u^+)^{2^\ast}\right] dx + 2 \vol{\Omega}. \QED
\]
\end{proof}

In the absence of a compact Sobolev imbedding, the main technical tool we use here to prove the \PS{} condition is the following concentration compactness principle of Lions \cite{MR834360,MR850686}.

\begin{proposition} \label{Proposition 4.2}
Let $\seq{u_j}$ be a sequence in $H^1_0(\Omega)$ such that $u_j \wto u \in H^1_0(\Omega)$ and $|\nabla u_j|^2\, dx \wstar \mu,\, |u_j|^{2^\ast}\, dx \wstar \nu$ in the sense of measures, where $\mu$ and $\nu$ are bounded nonnegative measures on $\closure{\Omega}$. Then there exist an at most countable index set $I$ and points $x_i \in \closure{\Omega},\, i \in I$ such that
\begin{enumroman}
\item $\nu = |u|^{2^\ast} dx + \sum_{i \in I} \nu_i\, \delta_{x_i}$, where each $\nu_i > 0$,
\item $\mu \ge |\nabla u|^2\, dx + \sum_{i \in I} \mu_i\, \delta_{x_i}$, where each $\mu_i > 0$ satisfies $(\nu_i)^{2/2^\ast} \le \mu_i/S$, in particular, $\sum_{i \in I} (\nu_i)^{2/2^\ast} < \infty$.
\end{enumroman}
\end{proposition}

\begin{proof}[Proof of Proposition \ref{Proposition 1.4}]
Let $0 < \kappa < \kappa^\ast$, $c \le M$, and $\seq{u_j} \subset H^1_0(\Omega)$ be a \PS{c} sequence. Since $\B \ge 0$,
\[
\norm{u_j}^2 \le 2 J_\eps(u_j) + \lambda \int_\Omega (u_j^+)^2\, dx + \frac{2 \kappa}{2^\ast} \int_\Omega (u_j^+)^{2^\ast}\, dx.
\]
We have
\[
\int_\Omega (u_j^+)^{2^\ast} dx = \O(1) \left(\norm{u_j} + 1\right)
\]
by Lemma \ref{Lemma 4.1}, and hence
\[
\int_\Omega (u_j^+)^2\, dx = \O(1)\, \big(\norm{u_j}^{2/2^\ast} + 1\big)
\]
by the H\"{o}lder inequality. It follows that $\seq{u_j}$ is bounded, and hence so is $(u_j^+)$.

Passing to a subsequence, $u_j^+$ converges to some $v \ge 0$ weakly in $H^1_0(\Omega)$, strongly in $L^p(\Omega)$ for $1 \le p < 2^\ast$ and a.e.\! in $\Omega$, and
\begin{equation} \label{4.1}
|\nabla u_j^+|^2\, dx \wstar \mu, \qquad (u_j^+)^{2^\ast} dx \wstar \nu
\end{equation}
in the sense of measures, where $\mu$ and $\nu$ are bounded nonnegative measures on $\closure{\Omega}$ (see, e.g., Folland \cite{MR1681462}). By Proposition \ref{Proposition 4.2}, there exist an at most countable index set $I$ and points $x_i \in \closure{\Omega},\, i \in I$ such that
\begin{equation} \label{4.2}
\nu = v^{2^\ast} dx + \sum_{i \in I} \nu_i\, \delta_{x_i}, \qquad \mu \ge |\nabla v|^2\, dx + \sum_{i \in I} \mu_i\, \delta_{x_i},
\end{equation}
where $\nu_i,\, \mu_i > 0$ and $(\nu_i)^{2/2^\ast} \le \mu_i/S$.

Let $\varphi : \R^N \to [0,1]$ be a smooth function such that $\varphi(x) = 1$ for $|x| \le 1$ and $\varphi(x) = 0$ for $|x| \ge 2$. For $i \in I$ and $\rho > 0$, set
\[
\varphi_{i,\rho}(x) = \varphi\left(\frac{x - x_i}{\rho}\right), \quad x \in \R^N,
\]
and note that $\varphi_{i,\rho} : \R^N \to [0,1]$ is a smooth function such that $\varphi_{i,\rho}(x) = 1$ for $|x - x_i| \le \rho$ and $\varphi_{i,\rho}(x) = 0$ for $|x - x_i| \ge 2 \rho$. Moreover, the sequence $(\varphi_{i,\rho}\, u_j^+)$ is bounded in $H^1_0(\Omega)$ and hence $J_\eps'(u_j)\, (\varphi_{i,\rho}\, u_j^+) \to 0$ as $j \to \infty$. We have
\begin{equation} \label{4.3}
J_\eps'(u_j)\, (\varphi_{i,\rho}\, u_j^+) \ge \int_\Omega \left[\varphi_{i,\rho}\, |\nabla u_j^+|^2 + u_j^+\, \nabla u_j^+ \cdot \nabla \varphi_{i,\rho} - \lambda\, \varphi_{i,\rho}\, (u_j^+)^2 - \kappa\, \varphi_{i,\rho}\, (u_j^+)^{2^\ast}\right] dx.
\end{equation}
By \eqref{4.1},
\[
\int_\Omega \varphi_{i,\rho}\, |\nabla u_j^+|^2\, dx \to \int_\Omega \varphi_{i,\rho}\, d\mu, \qquad \int_\Omega \varphi_{i,\rho}\, (u_j^+)^{2^\ast} dx \to \int_\Omega \varphi_{i,\rho}\, d\nu,
\]
and
\[
\int_\Omega \varphi_{i,\rho}\, (u_j^+)^2\, dx \to \int_\Omega \varphi_{i,\rho}\, v^2\, dx \le \int_{\Omega \cap B_{2 \rho}(x_i)} v^2\, dx.
\]
Denoting by $C$ a generic positive constant independent of $j$ and $\rho$,
\begin{multline*}
\abs{\int_\Omega u_j^+\, \nabla u_j^+ \cdot \nabla \varphi_{i,\rho}\, dx} \le \frac{C}{\rho} \left(\int_{\Omega \cap B_{2 \rho}(x_i)} (u_j^+)^2\, dx\right)^{1/2} \to \frac{C}{\rho} \left(\int_{\Omega \cap B_{2 \rho}(x_i)} v^2\, dx\right)^{1/2}\\[7.5pt]
\le C \left(\int_{\Omega \cap B_{2 \rho}(x_i)} v^{2^\ast}\, dx\right)^{1/2^\ast}
\end{multline*}
by the H\"{o}lder inequality. So letting $j \to \infty$ in \eqref{4.3} gives
\[
\int_\Omega \varphi_{i,\rho}\, d\mu - \kappa \int_\Omega \varphi_{i,\rho}\, d\nu \le \lambda \int_{\Omega \cap B_{2 \rho}(x_i)} v^2\, dx + C \left(\int_{\Omega \cap B_{2 \rho}(x_i)} v^{2^\ast}\, dx\right)^{1/2^\ast}.
\]
Letting $\rho \searrow 0$ and using \eqref{4.2} now gives $\mu_i \le \kappa \nu_i$, which together with $(\nu_i)^{2/2^\ast} \le \mu_i/S$ then gives $\nu_i \ge (S/\kappa)^{N/2}$. In particular, $I$ is a finite set since $\nu$ is a bounded measure.

Since $(u_j^+)$ is bounded, taking $u = u_j$ in Lemma \ref{Lemma 4.1} and passing to the limit using \eqref{4.1} gives
\[
\frac{\kappa}{N} \int_\Omega d\nu \le c + \vol{\Omega},
\]
which together with \eqref{4.2} and $c \le M$ then gives
\[
\kappa\, \sum_{i \in I} \nu_i \le N(M + \vol{\Omega}).
\]
Since each $\nu_i \ge (S/\kappa)^{N/2}$ and $\kappa < \kappa^\ast$, this implies that $I = \emptyset$ and hence
\begin{equation} \label{4.4}
\int_\Omega (u_j^+)^{2^\ast} dx \to \int_\Omega v^{2^\ast} dx.
\end{equation}

Since $\seq{u_j}$ is bounded, a further subsequence converges to some $u$ weakly in $H^1_0(\Omega)$, strongly in $L^p(\Omega)$ for $1 \le p < 2^\ast$, and a.e.\! in $\Omega$. We have
\begin{multline*}
\abs{(u_j - 1)_+^{2^\ast - 1}\, (u_j - u)} = \abs{(u_j^+)^{2^\ast - 1}\, (u_j^+ + u_j^- - u)}\\[7.5pt]
= \abs{(u_j^+)^{2^\ast} + (u_j^+)^{2^\ast - 1} - (u_j^+)^{2^\ast - 1}\, u} \le \left(2 - \frac{1}{2^\ast}\right) (u_j^+)^{2^\ast} + \frac{1}{2^\ast}\, (|u| + 1)^{2^\ast}
\end{multline*}
by Young's inequality, so
\[
\int_\Omega (u_j - 1)_+^{2^\ast - 1}\, (u_j - u)\, dx \to 0
\]
by \eqref{4.4} and the dominated convergence theorem. Then $u_j \to u$ in $H^1_0(\Omega)$ by a standard argument.
\end{proof}

Now we turn to the proof of Proposition \ref{Proposition 1.5}. If $u$ is a critical point of $J_\eps$, then $u$ is a weak solution of
\begin{equation} \label{4.5}
\left\{\begin{aligned}
- \Delta u & = - \frac{1}{\eps}\, \beta\left(\frac{u - 1}{\eps}\right) + \lambda\, (u - 1)_+ + \kappa\, (u - 1)_+^{2^\ast - 1} && \text{in } \Omega\\[10pt]
u & = 0 && \text{on } \bdry{\Omega}
\end{aligned}\right.
\end{equation}
and hence also a classical solution by elliptic regularity theory. By the maximum principle, either $u > 0$ in $\Omega$ or $u$ vanishes identically. We make use of the following $L^\infty$ bound, obtained in Perera and Silva \cite{MR2282829}, for weak solutions of the boundary value problem
\begin{equation} \label{4.6}
\left\{\begin{aligned}
- \Delta u & = h(x,u) && \text{in } \Omega\\[10pt]
u & > 0 && \text{in } \Omega\\[10pt]
u & = 0 && \text{on } \bdry{\Omega},
\end{aligned}\right.
\end{equation}
where $h$ is a Carath\'eodory function on $\Omega \times (0,\infty)$ satisfying
\begin{enumerate}
\item[$(h)$] $\exists\, r > N/2$ and $a \in L^r(\Omega)$ such that $h(x,t) \le a(x)\, t$ for a.a.\! $x \in \Omega$ and all $t > 0$.
\end{enumerate}

\newpage

\begin{proposition}[{\cite[Lemma A.1 \& Remark A.3]{MR2282829}}] \label{Proposition 4.3}
Suppose that $u \in H^1_0(\Omega)$ is a weak solution of problem \eqref{4.6}. Then there exists a constant $C > 0$, depending only on $\Omega$, $\norm[r]{a}$ and $\norm[2]{u}$, such that
\[
\norm[\infty]{u} \le C.
\]
\end{proposition}

\begin{proof}[Proof of Proposition \ref{Proposition 1.5}]
Let $u$ be a critical point of $J_\eps$ with $J_\eps(u) \le M$. Then $u$ is a weak solution of problem \eqref{4.6} with
\[
h(x,t) = - \frac{1}{\eps}\, \beta\left(\frac{t - 1}{\eps}\right) + \lambda\, (t - 1)_+ + \kappa\, (u(x) - 1)_+^\frac{4}{N-2}\, (t - 1)_+, \quad (x,t) \in \Omega \times (0,\infty).
\]
We have
\[
h(x,t) \le \left[\lambda + \kappa\, u(x)^\frac{4}{N-2}\right] t = \left[\lambda + \kappa\, w(x)^{4/N}\right] t \quad \forall x \in \Omega,\, t > 0,
\]
where $w = u^\frac{N}{N-2}$. We will show that if $0 < \kappa < \kappa_\ast$, then for
\[
r = \frac{N^2}{2\, (N - 2)} > \frac{N}{2},
\]
$\|w^{4/N}\|_r = \norm[2^\ast]{w}^{4/N}$ has a bound that depends only on $\norm[2^\ast]{u}$, $N$, $M$, $\vol{\Omega}$, $\kappa$ and $\lambda$. Since $\norm[2^\ast]{u}$ and $\norm[2]{u}$ have bounds that depend only on $N$, $M$, $\vol{\Omega}$, and $\kappa$ as noted in the paragraph before the statement of Proposition \ref{Proposition 1.5}, the conclusion will then follow from Proposition \ref{Proposition 4.3}.

By \eqref{1.4},
\begin{equation} \label{4.7}
S \left(\int_\Omega w^{2^\ast}\, dx\right)^{2/2^\ast} \le \int_\Omega |\nabla w|^2\, dx = \left(\frac{N}{N - 2}\right)^2 \int_\Omega u^\frac{4}{N-2}\, |\nabla u|^2\, dx.
\end{equation}
Noting that $u \in C^1(\closure{\Omega})$ by standard regularity arguments and testing \eqref{4.5} with $u^{2^\ast - 1}$ gives
\begin{equation}
\frac{N + 2}{N - 2}\, \int_\Omega u^\frac{4}{N-2}\, |\nabla u|^2\, dx \le \lambda \int_\Omega u^{2^\ast}\, dx + \kappa \int_\Omega w^2\, (u^+)^\frac{4}{N-2}\, dx.
\end{equation}
By the H\"{o}lder inequality and \eqref{1.5},
\begin{equation} \label{4.9}
\int_\Omega w^2\, (u^+)^\frac{4}{N-2}\, dx \le \left(\int_\Omega w^{2^\ast}\, dx\right)^{2/2^\ast} \left(\frac{N(M + \vol{\Omega})}{\kappa}\right)^{2/N}.
\end{equation}
Combining \eqref{4.7}--\eqref{4.9} gives
\[
\left[\left(1 - \frac{4}{N^2}\right) S - \kappa^\frac{N-2}{N}\, \big[N(M + \vol{\Omega})\big]^{2/N}\right] \norm[2^\ast]{w}^2 \le \lambda \norm[2^\ast]{u}^{2^\ast}.
\]
Since $\kappa < \kappa_\ast$, this gives the required bound for $\norm[2^\ast]{w}$.
\end{proof}

\section{Existence of a ground state}

In this section we prove Theorem \ref{Theorem 1.2}. We have
\[
J_\eps(u) \ge \int_\Omega \left[\frac{1}{2}\, |\nabla u|^2 - \left(\frac{\lambda}{2} + \frac{\kappa}{2^\ast}\right) |u|^{2^\ast}\right] dx \ge \frac{1}{2} \norm{u}^2 - \left(\frac{\lambda}{2} + \frac{\kappa}{2^\ast}\right) S^{- \frac{N}{N-2}} \norm{u}^{2^\ast}
\]
by \eqref{1.4}, so there exists a constant $\rho > 0$, depending only on $N$, $\kappa$ and $\lambda$, such that
\[
\norm{u} \le \rho \implies J_\eps(u) \ge \frac{1}{3} \norm{u}^2.
\]
Moreover,
\[
J_\eps(u) \le \int_\Omega \left[\frac{1}{2}\, |\nabla u|^2 + 1 - \frac{\kappa}{2^\ast}\, (u^+)^{2^\ast}\right] dx
\]
and hence there exists a function $u_0 \in H^1_0(\Omega)$ such that $J_\eps(u_0) < 0 = J_\eps(0)$. So the class of paths
\[
\Gamma_\eps = \set{\gamma \in C([0,1],H^1_0(\Omega)) : \gamma(0) = 0,\, J_\eps(\gamma(1)) < 0}
\]
is nonempty and
\begin{equation} \label{5.1}
c_\eps := \inf_{\gamma \in \Gamma_\eps}\, \max_{u \in \gamma([0,1])}\, J_\eps(u) \ge \frac{\rho^2}{3}.
\end{equation}
Since $\B((t - 1)/\eps) \le \goodchi_{\set{t > 1}}$ for all $t$, $J_\eps(u) \le J(u)$ for all $u \in H^1_0(\Omega)$. So $\Gamma \subset \Gamma_\eps$ and
\[
c_\eps \le \max_{u \in \gamma([0,1])}\, J_\eps(u) \le \max_{u \in \gamma([0,1])}\, J(u) \quad \forall \gamma \in \Gamma,
\]
and hence
\begin{equation} \label{5.2}
c_\eps \le c.
\end{equation}

\begin{lemma} \label{Lemma 5.1}
Given $\lambda_\ast > \lambda_1$, there exists a constant $M_{\lambda_\ast} > 0$, depending only on $\lambda_\ast$ and $\Omega$, such that $c \le M_{\lambda_\ast}$ for all $\lambda \ge \lambda_\ast$ and $\kappa > 0$.
\end{lemma}

\begin{proof}
Let $\varphi_1 > 0$ be an eigenfunction associated with $\lambda_1$. Then
\[
J(t \varphi_1) = \int_\Omega \left[\frac{t^2}{2}\, |\nabla \varphi_1|^2 + \goodchi_{\set{t \varphi_1 > 1}}(x) - \frac{\lambda}{2}\, (t \varphi_1 - 1)_+^2 - \frac{\kappa}{2^\ast}\, (t \varphi_1 - 1)_+^{2^\ast}\right] dx \to - \infty
\]
as $t \to \infty$. Taking $t_0 > 0$ so large that $J(t_0 \varphi_1) < 0$ and reparametrizing the path \linebreak $[0,t_0] \to H^1_0(\Omega),\, t \mapsto t \varphi_1$ gives a path in $\Gamma$ on which $J$ is less than or equal to
\[
M_\lambda := \sup_{t \ge 0}\, \int_\Omega \left[\frac{\lambda_1}{2}\, t^2 \varphi_1^2 + 1 - \frac{\lambda}{2}\, (t \varphi_1 - 1)_+^2\right] dx,
\]
which is finite for $\lambda > \lambda_1$, so $c \le M_\lambda$. Clearly, $M_\lambda \le M_{\lambda_\ast}$ if $\lambda \ge \lambda_\ast$.
\end{proof}

\begin{proof}[Proof of Theorem \ref{Theorem 1.2}]
Let $\lambda_\ast > \lambda_1$, let $M_{\lambda_\ast} > 0$ be as in Lemma \ref{Lemma 5.1}, let $\kappa^\ast > \kappa_\ast > 0$ be as in Propositions \ref{Proposition 1.4} and \ref{Proposition 1.5}, respectively, with $M = M_{\lambda_\ast}$, and let $\lambda \ge \lambda_\ast$ and $0 < \kappa < \kappa_\ast$. Then $c_\eps \le c \le M_{\lambda_\ast}$ by \eqref{5.2} and Lemma \ref{Lemma 5.1}, so $J_\eps$ satisfies the \PS{c_\eps} condition by Proposition \ref{Proposition 1.4}, and hence $J_\eps$ has a critical point $u_\eps$ at the level $c_\eps$ by a standard argument. Take a sequence $\eps_j \searrow 0$ and let $u_j = u_{\eps_j},\, c_j = c_{\eps_j}$. Then $\seq{u_j}$ is bounded in $H^1_0(\Omega) \cap L^\infty(\Omega)$ by Proposition \ref{Proposition 1.5} and the paragraph preceding the proposition. Hence a renamed subsequence converges, uniformly on $\closure{\Omega}$ and strongly in $H^1_0(\Omega)$, to a locally Lipschitz continuous function $u \in H^1_0(\Omega) \cap C(\closure{\Omega}) \cap C^1(\closure{\Omega} \setminus \bdry{\set{u > 1}}) \cap C^2(\Omega \setminus \bdry{\set{u > 1}})$ by Theorem \ref{Theorem 1.3} \ref{Theorem 1.3.i} and \ref{Theorem 1.3.ii}. Moreover, $\limsup c_j \ge \rho^2/3 > 0$ by \eqref{5.1} and hence $u$ is nontrivial by Theorem \ref{Theorem 1.3} \ref{Theorem 1.3.iii}.

Also by Theorem \ref{Theorem 1.3}, $u$ satisfies the equation $- \Delta u = \lambda\, (u - 1)_+ + \kappa\, (u - 1)_+^{2^\ast - 1}$ in the classical sense in $\Omega \setminus \bdry{\set{u > 1}}$, the free boundary condition $|\nabla u^+|^2 - |\nabla u^-|^2 = 2$ in the generalized sense \eqref{1.2}, and vanishes on $\bdry{\Omega}$. In particular, $u \in \M$ and hence
\[
\inf_\M\, J \le J(u).
\]
Since we also have
\[
J(u) \le \liminf c_j \le \limsup c_j \le c \le \inf_\M\, J
\]
by Theorem \ref{Theorem 1.3} \ref{Theorem 1.3.iii}, \eqref{5.2}, and Proposition \ref{Proposition 2.1}, then
\[
J(u) = \inf_\M\, J = c,
\]
so $u$ is a ground state at the level $c \ge \rho^2/3 > 0$. Hence $u$ is a mountain pass point of $J$ by Proposition \ref{Proposition 2.1} and $u$ is nondegenerate by Theorem \ref{Theorem 1.1}. The rest then follows from the last part of Theorem \ref{Theorem 1.3}.
\end{proof}

\def\cdprime{$''$}

\end{document}